\documentclass{amsart}
\usepackage{amssymb}
\usepackage{amsfonts}
\usepackage{mathrsfs}
\usepackage[all]{xy}
\usepackage{color}
\usepackage{cite}
\usepackage[numbered]{algo}
\usepackage{url}

\DeclareMathOperator{\Aut}{Aut}

\DeclareMathOperator{\Cl}{Cl}

\DeclareMathOperator{\disc}{disc}

\DeclareMathOperator{\Gal}{Gal}

\DeclareMathOperator{\PSL}{PSL}

\DeclareMathOperator{\SL}{SL}

\DeclareMathOperator{\Surj}{Surj}

\DeclareMathOperator{\tame}{tame}

\newcommand{\EE}{\mathbb{E}}
\newcommand{\FF}{\mathbb{F}}

\newcommand{\NN}{\mathbb{N}}

\newcommand{\QQ}{\mathbb{Q}}
\newcommand{\RR}{\mathbb{R}}
\newcommand{\ZZ}{\mathbb{Z}}

\newcommand{\mfa}{\mathfrak{a}}

\newcommand{\mfo}{\mathfrak{o}}

\newcommand{\mfP}{\mathfrak{P}}
\newcommand{\mfp}{\mathfrak{p}}
\newcommand{\mfq}{\mathfrak{q}}
\newcommand{\mfz}{\mathfrak{z}}
\newcommand{\msP}{\mathscr{P}}

\newtheorem{thm}{Theorem}
\newtheorem{Lemma}[thm]{Lemma}

\newtheorem{conj}[thm]{Conjecture}
\newtheorem{prop}[thm]{Proposition}
\newtheorem{heur}[thm]{Heuristic}

\newtheorem{Claim}[thm]{Claim}
\newtheorem{cond}[thm]{Conditions}

\newtheorem{defn}[thm]{Definition}
\newtheorem{rem}[thm]{Remark}
\newtheorem{ex}[thm]{Example}
\newtheorem{notn}[thm]{Notation}

\algorithmcounter{thm}
\numberwithin{equation}{section}
\numberwithin{thm}{section}

\begin{document}

\markboth{Simon Rubinstein-Salzedo}
{$A_4$ Invariants and the Cohen-Lenstra Heuristics}

\title{Invariants for $A_4$ Fields and the Cohen-Lenstra Heuristics}

\author{Simon Rubinstein-Salzedo}

\address{Department of Mathematics, Dartmouth College, Hanover, New Hampshire 03755, USA}
\email{simon.rubinstein-salzedo@dartmouth.edu}

\maketitle

\begin{abstract}
This article discusses deviations from the Cohen-Lenstra heuristics when roots of unity are present. In particular, we propose an explanation for the discrepancy between the observed number of cyclic cubic fields whose 2-class group is $C_2\times C_2$ and the number predicted by the Cohen-Lenstra heuristics, in terms of an invariant living in a quotient of the Schur multiplier group. We also show that, in some cases, the definition of the invariant can be simplified greatly, and we compute the invariant when the cubic field is ramified at exactly one prime, up to $10^8$.
\end{abstract}

\section{Introduction} \label{introd}

Cohen and Lenstra in \cite{CL83} were interested in studying the distribution of class groups of quadratic fields, and perhaps, abelian extensions of $\QQ$ more generally. These heuristics were extended by Cohen and Martinet in \cite{CM87} and \cite{CM90} to fields of more general type.

%Since class groups are finite abelian groups, we can attempt to understand them by understanding their Sylow $p$-subgroups. Assuming that the distributions of the Sylow $p$-subgroups are independent for different $p$, we can then patch together the distribution of class groups from the distribution of the Sylow $p$-subgroups for all $p$. (The independence of distributions at each prime is conjectural, but it is well-supported by numerical data.) 

We will restrict ourselves to looking at the distribution of the Sylow $p$-subgroups of the class groups of number fields. If $K$ is a number field, we let $\Cl_p(K)$ denote the Sylow $p$-subgroup of $\Cl(K)$.

\begin{notn} Throughout this article, we use the notation $C_n$ to denote the cyclic group of order $n$. \end{notn}

In the case of quadratic fields, Cohen and Lenstra made the following conjecture:

\begin{conj}[Cohen-Lenstra] Let $p$ be an odd prime. Let $D^\pm(X)$ denote the set of real (respectively imaginary) quadratic fields $K$ with $|\disc(K)|<X$. Let $A$ be a finite abelian $p$-group. Then \[\lambda^\pm(A) = \lim_{X\to\infty} \frac{\#\{K\in D^\pm(X):\Cl_p(K)\cong A\}}{\# D^\pm(X)}\] exists, and we have \[\lambda^+(A)=c^+|\Aut(A)|^{-1}\times|A|^{-1},\qquad \lambda^-(A)=c^-|\Aut(A)|^{-1}\] for certain explicit constants $c^+$ and $c^-$, which are independent of $A$. \end{conj}

The statement of this conjecture suggests many further questions. One such question is why we need to restrict to the case of an odd prime $p$. The reason is that the 2-torsion in the class group is controlled by genus theory. If $K$ is a quadratic field and $r$ is the number of primes dividing $\disc(K)$, then the 2-torsion of $\Cl(K)$ is isomorphic to $C_2^{r-1}$ or $C_2^{r-2}$ (see \cite{FT93}, Corollary 1 to Theorem 39, for a more precise statement and a proof). In particular, the 2-torsion in quadratic fields is rarely equal to 0 and can easily become arbitrarily large as we vary by discriminant.

Of course, we need not lose interest in class groups as soon as we step beyond quadratic fields: we could ask the same question for fields of other types. Furthermore, in this case, the 2-power torsion will not necessarily be governed by genus theory, so we might also allow $p$ to be 2. So, we could make the following guess, by attempting to apply the Cohen-Lenstra heuristics to situations for which we have no \textit{a priori} reason for believing that they are appropriate.

\begin{heur}[Proto-Cohen-Lenstra Heuristics] Let $n$ be a positive integer, and let $G$ be a transitive permutation group on a set of size $n$. Furthermore, let $(r_1,r_2)$ be a signature, with $r_1+2r_2=n$. Let $D(X)$ be the set of number fields $K$ the absolute value of whose discriminant is less than $X$, and so that the Galois group $\Gal(K^\sharp/\QQ)$ of the Galois closure of $K$ is isomorphic to $G$, and $K$ has $r_1$ real embeddings and $r_2$ pairs of complex conjugate embeddings. Finally, let $p$ be a prime not dividing $|G|$ and let $A$ be a finite abelian $p$-group. Then \[\lim_{X\to\infty} \frac{\#\{K\in D(X):\Cl_p(K)\cong A\}}{\# D(X)}\] exists, and is inversely proportional to $|\Aut(A)|\times |A|^{r_1+r_2-1}$. \end{heur}

However, the proto-Cohen-Lenstra heuristics sometimes fail for silly reasons. Here is an example:

\begin{Lemma} Let $n=3$ and $G=C_3$ in the above heuristics. If $p\equiv 2\pmod 3$, then the $p$-rank of $K$, $r_p(K):=\dim_{\FF_p}(\Cl(K)/p\Cl(K))$, is even for all such $K$. \end{Lemma}

\begin{proof} The Galois group $C_3=\{1,\sigma,\sigma^2\}$ of $K$ over $\QQ$ acts on $\Cl_p(K)$. The number of ideal classes of order $p$ is $p^{r_p(K)}-1$, which is congruent to $0\pmod 3$ if and only if $r_p(K)$ is even. Hence, if $r_p(K)$ is odd, then there must be a nontrivial $p$-torsion ideal class $C$ fixed by the Galois action. In particular, there is a nonprincipal ideal $\mfa\in C$ of order $p$ so that $\mfa$, $\mfa^\sigma$, and $\mfa^{\sigma^2}$ are all in the same ideal class. We now show that their product $\mfa^{1+\sigma+\sigma^2}$ is equal to the principal ideal $(N\mfa)\mfo_K$.  For any $a\in\mfa$, $Na=a^{1+\sigma+\sigma^2}\in\mfa^{1+\sigma+\sigma^2}$, and the $Na$ generate $(N\mfa)\mfo_K$ as an $\mfo_K$-module, so $(N\mfa)\mfo_K\subset\mfa^{1+\sigma+\sigma^2}$. Now, the norms of both ideals $\mfa^{1+\sigma+\sigma^2}$ and $(N\mfa)\mfo_K$ are $(N\mfa)^3$. Hence, they are equal. Furthermore, $(N\mfa)\mfo_K$ is principal, since it is generated by the element $N\mfa$. Hence $C^3=1$ in $\Cl_p(K)$. But this is impossible, as $3\nmid p$. \end{proof}

We can patch the proto-Cohen-Lenstra heuristics by excluding those $A$ that are ruled out by this Lemma and related ones. Furthermore, as the proof of the Lemma hints, for those $A$ that are allowable, we need the automorphisms to be compatible with the Galois action, in the case that $K$ is actually a Galois number field, as was first suggested by Cohen and Martinet in \cite{CM87}. In particular, if $G=C_\ell$, then we need $A$ to be a $\ZZ[\zeta_\ell]$-module. This suggests the following refinement of the proto-Cohen-Lenstra heuristics, at least in the case where $n=\ell$ is a prime, and $G=C_\ell$:

\begin{heur}[Refined Cohen-Lenstra Heuristics] Let $\ell$ be an odd prime, and let $G=C_\ell$. Let $D(X)$ be the set of $C_\ell$ number fields with absolute discriminant less than $X$. (Such fields are necessarily totally real.) Also, let $p$ be a prime different from $\ell$ and $A$ an abelian $p$-group with the structure of a $\ZZ[\zeta_\ell]$-module. Then \[\lim_{X\to\infty}\frac{\# \{K\in D(X):\Cl_p(K)\cong A\}}{\# D(X)}\] exists, and is inversely proportional to $|\Aut_{\ZZ[\zeta_\ell]}(A)|\times |A|^{\ell-1}$. \end{heur}

Another, more modern and sometimes cleaner, way to interpret the refined Cohen-Lenstra heuristics is based on the following idea from probability theory. Let $\mu$ be a probability distribution on $\RR$. Define the $k^\text{th}$ moment of $\mu$ to be \[a_k=\int_{-\infty}^\infty x^k\; d\mu.\] From knowing the sequence of moments $a_1,a_2,\ldots$, it is possible to reconstruct $\mu$ under fairly mild hypotheses. To be more precise, define the moment generating function to be the power series \[A(x)=\sum_{k=1}^\infty a_k\frac{x^k}{k!}.\] Then assuming that $A(x)$ has positive radius of convergence, $\mu$ is the only probability distribution having moment generating function $A(x)$. (See \cite{Bill95}, Chapter 30, for a proof.)

In the context at hand, we can define an analogue of a moment for a probability distribution $f$ of finite abelian $p$-groups as follows. Fix a finite abelian $p$-group $A$, and look at the expected number of surjections (in whichever category is appropriate) from an $f$-random finite abelian $p$-group $X$ to $A$. This number behaves as the ``$A^\text{th}$ moment of $X$.'' We expect that, just as in the situation for classical moments of probability distributions, these $A^\text{th}$ moments of $X$ determine $f$, assuming that $f$ is fairly well-behaved. (We will not prove that these moments determine $f$, although we expect the proof not to be too difficult, since we are only using this example for motivation here.)

We now put this in proper context. If $A$ is a finite abelian $p$-group that also has the structure of a $\ZZ[\zeta_\ell]$-module, then we would like to understand the number \[\EE(\#\Surj_{\ZZ[\zeta_\ell]}(\Cl(K),A)).\] Here $\Surj$ is the set of surjective maps, and $\EE$ denotes the expected value. Let us look at the case of $n=3$ and $G=C_3$. Then the refined Cohen-Lenstra heuristics imply that \begin{equation} \label{CLrank} \lim_{X\to\infty} \frac{1}{\# D(X)}\sum_{K\in D(X)} p^{r_p(K)}=\begin{cases} \left(1+\frac{1}{p}\right)^2 & p\equiv 1\pmod 3,\\ 1+\frac{1}{p^2} & p\equiv 2\pmod 3.\end{cases}\end{equation}

In the case when $A=\FF_p[\zeta_3]$ as a $\ZZ[\zeta_3]$-module (so that the underlying abelian group structure for $A$ is $C_p$ if $p\equiv 1\pmod 3$ and $C_p\times C_p$ if $p\equiv 2\pmod 3$), then \begin{equation*} \EE(\#\Surj_{\ZZ[\zeta_3]}(\Cl(K),A))=\lim_{X\to\infty} \frac{1}{\# D(X)}\sum_{K\in D(X)} (p^{r_p(K)}-1).\end{equation*} So, in particular, if $p=2$, we would expect the number of surjections from the class group of a random $C_3$ field to $C_2\times C_2$ to be $1/4$. As we shall see shortly, however, this appears not to be the case.

The goal of this article is to attempt to explain this discrepancy, noted by Malle in \cite{Malle08} in the Cohen-Lenstra heuristics when roots of unity are present. We will focus on one case: that of $C_3$ fields whose class groups surject onto $C_2\times C_2$. It will turn out to be more convenient for us not to work directly with the $C_3$ field; instead, class field theory associates to a surjection from the class group of a $C_3$ field to $C_2\times C_2$ an $A_4$ field; we work with this $A_4$ field instead.

In \S\ref{Malle}, we present the results of Malle's computations and what the modifications to the Cohen-Lenstra heuristics appear to be. In order to introduce our new results and explanations, we discuss the Schur multiplier and a variant of it called the reduced Schur multiplier; this is done in \S \ref{Schur}.

The new material begins in \S \ref{nf-invts}. Here we present an invariant associated to an $A_4$ field which is expected to explain the discrepancy in (\ref{extrasurj}). In \S \ref{sec:comp}, we present an algorithm to compute the invariant, and we show that in certain circumstances, the invariant has a simple interpretation as the parity of the class group of a certain field. In \S \ref{sec:samp}, we perform an explicit computation of the invariant for the smallest $A_4$ field. Finally, we end this article with Table \ref{invtdata}, which summarizes the data collected from $377529$ fields, and a guess about a possible secondary term.

This article is adapted from the author's PhD thesis \cite{RSthesis}. A reader who wishes for a slightly more leisurely exposition may prefer to read Chapters 2 and 3 of the thesis instead.

\section{Malle's Computations} \label{Malle}

\begin{notn}We use the following notation: for $q,k\in\NN$, let \[(q)_k=\prod_{i=1}^k (1-q^{-1}),\qquad (q)_\infty=\prod_{i=1}^\infty (1-q^{-1}).\] \end{notn}

After performing many tests, Malle proposed a list of cases in which the Cohen-Lenstra heuristics are expected to fail. In particular, when $p=2$, they should always fail. In the case of $C_3$ fields, the Cohen-Lenstra heuristics predict that the Sylow 2-subgroup of the class group should be isomorphic to $C_2\times C_2$ with probability \[\frac{1}{12}\frac{(4)_\infty}{(4)_1}\approx .0765.\] Instead, in his sample of over 16 million fields, he finds that the actual probability is closer to $.13$, nearly twice as large as expected. Similarly, equation $(\ref{CLrank})$ does not seem to hold when $p=2$: equation $(\ref{CLrank})$ predicts that the average size of the maximal elementary abelian 2-subgroup of $\Cl(K)$ be $\frac{5}{4}$, but Malle's computations suggest that the correct number is $\frac{3}{2}$.

In terms of expected number of surjections, it appears that \begin{equation}\EE(\#\Surj_{\ZZ[\zeta_3]}(\Cl(K),C_2\times C_2))=1/2, \label{extrasurj} \end{equation} rather than $1/4$, as mentioned in the previous section.

In general, Malle expects the Cohen-Lenstra heuristics to fail at the prime $p$ when the ground field contains $p^\text{th}$ roots of unity. In this case, he expects that if $A$ is a nontrivial abelian $p$-group, then $\Cl_p(K)\cong A$ more often than the Cohen-Lenstra heuristics predict.

\begin{rem} For $C_3$ fields, the Cohen-Lenstra prediction also fails for $p=3$, since the 3-torsion in the class group is governed by genus theory, just as in the case of $p=2$ for quadratic fields. More generally, the Cohen-Lenstra predictions at a prime $p$ do not hold for fields with Galois group $G$ if $p$ divides $|G|$ because in this case there is a contribution coming from genus theory. In this article, we are not especially interested in the failure for that reason, since genus theory is well-understood. Thus, we will only be concerned with deviations due to the existence of $p^\text{th}$ roots of unity. \end{rem}

\section{Schur Multipliers and Variants} \label{Schur}

Our proposed correction to the Cohen-Lenstra heuristics in the presence of roots of unity can be described in terms of the reduced Schur multiplier. We first recall the definition of the Schur multiplier, then move on to the reduced Schur multiplier.

\begin{defn} Let $G$ be a group. \begin{enumerate} \item The Schur multiplier group of $G$ is defined to be the second homology group $H_2(G,\ZZ)$. \item A central extension of $G$ is a short exact sequence \[0\to A\to\widetilde{G}\to G\to 1,\] where $A$ is an abelian group, and $A$ is contained in the center of $\widetilde{G}$. \item A stem extension of $G$ is a central extension \[0\to A\to\widetilde{G}\to G\to 1\] so that $A$ is contained in the intersection of the center of $\widetilde{G}$ and the commutator subgroup of $\widetilde{G}$. \end{enumerate} \end{defn}

If $G$ is finite, there is a stem extension $\widetilde{G}$ of maximal order; in fact, there may be more than one of maximal order, and such $\widetilde{G}$ need not be isomorphic. However, as $\widetilde{G}$ varies over maximal stem extensions, the corresponding $A$ are all isomorphic, and they are isomorphic to the Schur multiplier group $H_2(G,\ZZ)$. We call a stem extension of the form \[0\to H_2(G,\ZZ)\to\widetilde{G}\to G\to 1\] a Schur cover. If, in addition, $G$ is a perfect group (i.e., $G=[G,G]$ is its own commutator subgroup), then there is a unique such group $\widetilde{G}$.

Suppose $G$ is a finite group. Then $H_2(G,\ZZ)$ is a finite group all of whose elements have order dividing the order of $G$. Also, for a prime $p$, the Sylow $p$-subgroup of $H_2(G,\ZZ)$ is trivial if the Sylow $p$-subgroup of $G$ is cyclic. More information about the Schur multiplier, including many examples of the Schur multiplier group $H_2(G,\ZZ)$ for various finite groups $G$, can be found, for instance, in \cite{Karp87}. %For convenience, we provide a few examples of Schur multiplier groups.

%\begin{prop} \begin{enumerate} \item (Schur 1907, also Corollary 2.2.12 in \cite{Karp87}) Let $G$ be the finite abelian group \[G\cong C_{n_1}\times C_{n_2}\times\cdots\times C_{n_k},\] with $n_{i+1}\mid n_i$ for $1\le i\le k-1$. Let $C_n^{(m)}$ denote the direct product of $m$ copies of $C_n$. Then \[H_2(G,\ZZ)\cong C_{n_2}\times C_{n_3}^{(2)}\times\cdots\times C_{n_k}^{(k-1)}.\] \item Let $G=C_2\times C_2$. Then $H_2(G,\ZZ)\cong C_2$. There are two maximal stem extensions of $G$: \[0\to C_2\to D_8\to C_2\times C_2\to 0\] and \begin{equation} \label{Sylow2} 0\to C_2\to Q_8\to C_2\times C_2\to 0.\end{equation} \item Let $G=A_n$ be the alternating group on $n$ letters. Then \[H_2(G,\ZZ)=\begin{cases} 1 & n\le 3, \\ C_2 & n\ge4 \text{ and } n\neq 6,7, \\ C_6 & n=6,7.\end{cases}\] If $n=4$, then we have $A_4\cong\PSL_2(\FF_3)$, and the unique maximal stem extension of $A_4$ is \begin{equation}0\to C_2\to\SL_2(\FF_3)\to A_4\to 1.\label{A4seq}\end{equation} Furthermore, the sequence (\ref{Sylow2}) is the sequence of Sylow 2-subgroups of (\ref{A4seq}). If $n=5$, we have $A_5\cong\PSL_2(\FF_5)$, and the unique maximal stem extension of $A_5$ is \[0\to C_2\to\SL_2(\FF_5)\to A_5\to 1.\] We write $\widetilde{A}_n$ for the maximal stem extension of $A_n$. \item Let $G=S_n$ be the symmetric group on $n$ letters. Then \[H_2(G,\ZZ)=\begin{cases} 1 & n\le 3, \\ C_2 & n\ge 4.\end{cases}\] For $n\ge 4$, there are two nonisomorphic double covers of $S_n$. \end{enumerate} \end{prop}

In fact, what we really need is not the full Schur multiplier group, but a certain quotient of it, associated to a certain union of conjugacy classes of $G$. To this end, fix a union of conjugacy classes $c\subset G$. Let \[0\to H_2(G,\ZZ)\to\widetilde{G}\to G\to 1\] be a Schur cover. Suppose $x\in c$ and $y\in G$ commute. Lift $x$ and $y$ to $\widetilde{x}$ and $\widetilde{y}$, respectively, in $\widetilde{G}$. (This can be done in multiple ways; choose one arbitrarily.) Then the commutator $[\widetilde{x},\widetilde{y}]_{\widetilde{G}}$ lies in $H_2(G,\ZZ)$, and this element is independent of the choice of lifts. Call this element $\langle x,y\rangle_{\widetilde{G}}$. Let $Q_c$ denote the subgroup of $H_2(G,\ZZ)$ generated by all the $\langle x,y\rangle_{\widetilde{G}}$'s.

\begin{defn} The reduced Schur multiplier of a pair $(G,c)$ is the quotient \[H_2(G,c,\ZZ)=H_2(G,\ZZ)/Q_c.\] A reduced Schur cover of $(G,c)$ is the quotient $\widetilde{G}_c=\widetilde{G}/Q_c$. \end{defn}

A reduced Schur cover is a largest stem extension of $G$ so that $c$ lifts bijectively to a union of conjugacy classes $\widetilde{c}\subset\widetilde{G}_c$.

\begin{rem} We will tend to be slightly sloppy with our terminology when referring to reduced Schur covers. In the future, when we refer to a reduced Schur cover $\widetilde{G}_c$, we shall assume that it comes packaged with a union of conjugacy classes $\widetilde{c}\subset\widetilde{G}_c$ which bijects onto $c$, even when no explicit choice of $\widetilde{c}$ is provided. \end{rem}

\begin{ex} Suppose $G=A_5$. If $c$ is the conjugacy class of 3-cycles, then $H_2(G,c,\ZZ)=H_2(G,\ZZ)\cong C_2$, and the corresponding extension is \[0\to C_2\to\widetilde{A}_5\to A_5\to 1.\] To see this, note that if $x\in c$ is a 3-cycle, then the only elements of $A_5$ commuting with $x$ are $1,x,x^2$. Hence $Q_c$ is trivial. However, if $c$ is the conjugacy class of $(12)(34)$, then $H_2(G,c,\ZZ)$ is trivial. To see this, we work with matrices, since $A_5\cong\PSL_2\FF_5$ and $\widetilde{A}_5\cong\SL_2\FF_5$. We must show that $-I\in Q_c$, and we do this by choosing a suitable $x$ and $y$. Take \[\widetilde{x}=\begin{pmatrix} 3 & 3 \\ 4 & 1 \end{pmatrix},\qquad \widetilde{y}=\begin{pmatrix} 1 & 1 \\ 3 & 4 \end{pmatrix},\] both matrices in $\SL_2\FF_5$ whose images in $\PSL_2\FF_5$ commute. Then their commutator is $-I$, so $-I\in Q_c$, so $H_2(G,c,\ZZ)$ is trivial. \end{ex}

\section{Invariants for Number Fields} \label{nf-invts}

We now introduce the invariant of Ellenberg and Venkatesh as described in \cite{EV10}. This will be an attempt to explain (\ref{extrasurj}), as follows: Ordinarily, we would expect the right-hand side of (\ref{extrasurj}) to be $1/4$, but in this case, it is twice as large as we anticipate. To each surjection $\varphi:\Cl(K)\twoheadrightarrow C_2\times C_2$ of $\ZZ[\zeta_3]$-modules, we associate an invariant $\mfz(\varphi)\in\{0,1\}$. We then hope that \[\EE(\#\{\varphi\in\Surj_{\ZZ[\zeta_3]}(\Cl(K),C_2\times C_2):\mfz(\varphi)=0\})=\frac{1}{4}\] and \[\EE(\#\{\varphi\in\Surj_{\ZZ[\zeta_3]}(\Cl(K),C_2\times C_2):\mfz(\varphi)=1\})=\frac{1}{4}.\] (For convenience, we will overuse the notation $\mfz$ a little bit: sometimes, we will write $\mfz(\varphi)$ to denote the invariant of a surjection, and sometimes we will write $\mfz(\rho)$ to denote the invariant of a field corresponding to a representation $\rho:G_K\to C_2\times C_2$.)

The motivation for this comes from the case of function fields, which was studied by Ellenberg, Venkatesh, and Westerland in \cite{EVW}. In this case, the extensions are parametrized by a Hurwitz space, which may have several connected components. On each connected component, the number of extensions agrees (asymptotically) with the Cohen-Lenstra predictions, but there is a discrepancy when there are multiple connected components. This is discussed at the end of \cite{EVW2}. In the number field case, we have no Hurwitz space to parametrize the extensions, but we are left with a vestige of the connected components, which are given in terms of the Schur multiplier.

We now consider the following scenario, which is a modification of that considered by Ellenberg and Venkatesh in \cite{EV10}. Let $K$ be a number field or function field, let $G$ be a finite group, and let $c=c_1\cup\cdots\cup c_r\subset G$ be a union of conjugacy classes. Then, we assume that the following conditions hold:

\begin{cond} \label{conditions1} \begin{enumerate} \item $G$ has trivial center. \item $c$ generates $G$. \item If $n$ is prime to the order of an element $g\in c$, then $g^n\in c$. \end{enumerate} \end{cond}

\begin{ex} These conditions hold if $G=A_4$ and $c$ is the union of the two conjugacy classes of 3-cycles. They also hold if $G=A_5$ and $c$ either the conjugacy class of 3-cycles or the conjugacy class of a product of two disjoint 2-cycles. \end{ex}

\begin{Lemma}[Ellenberg-Venkatesh] \label{EVlemma} Let $K$ be a totally real number field, let $G$ be a finite group, and let $c$ be a union of conjugacy classes of $G$, satisfying (3) from Conditions \ref{conditions1} above, and let $\rho:G_K\to G$ be a homomorphism so that \begin{enumerate} \item $\rho$ is trivial at all infinite places. \item $\rho$ is tamely ramified, and the image of each inertia group $I_{\mfP}$ in $G_K$ is a cyclic subgroup contained in $c\cup\{1\}$. \end{enumerate} Furthermore, we assume that $2H_2(G,c,\ZZ)=0$. Then $\rho$ lifts to an extension $\widetilde{\rho}:G_K\to\widetilde{G}_c$ which is trivial at all infinite places and tamely ramified. \end{Lemma}

%\begin{rem} The conclusion of the Lemma is sometimes still valid even when the hypotheses are not all satisfied. In particular, if $G=C_2\times C_2$ and $c$ consists of just the identity, the conclusion still holds, even though $c$ does not generate $G$. This particular case will show up again shortly. \end{rem}

We may now define the invariant $\mfz(\rho)\in H_2(G,c,\ZZ)$. Suppose that $H_2(G,c,\ZZ)\neq 0$. For each finite place $v$ of $K$, let $\mfp_v$ be the corresponding prime and $k_v$ the residue field at $v$, and let $q_v$ be the size of $k_v$. Let $I_v$ be the inertia group of $G_K$ at $v$, and fix an element $\pi\in\mfp_v-\mfp_v^2$. We have a map $I_v\to k_v^\times$, given by $\sigma\mapsto \sigma(\pi)/\pi\mod{\mfp_v}$. Let $g_v$ be any inverse image of $-1$ so that $g_v$ topologically generates a subgroup of $I_v^{\tame}$ of index $\frac{q_v-1}{2}$. Now, each $x\in c$ is the image of a unique $x^\ast\in\widetilde{c}$.

\begin{defn} \label{inv-def} The invariant $\mfz(\rho)$ is defined to be \begin{equation}\mfz(\rho)=\prod_{v\text{ finite}}\widetilde{\rho}(g_v)(\rho(g_v)^\ast)^{-1}\in H_2(G,c,\ZZ).\label{invtdef}\end{equation} \end{defn}

This invariant is independent of choice of $\widetilde{\rho}$, $g_v$, and $I_v$. The independence is explained in \cite{EV10}.

\begin{defn} If $L/K$ is a Galois extension with group $G$, we say that all ramification of $L/K$ is of type $c$ if $L/K$ is tamely ramified, and for each prime $\mfP$ of $L$, either $\mfP$ is unramified, or else a generator of the inertia group at $\mfP$ is contained in $c$. \end{defn}

We consider Galois extensions $L/K$ with Galois group $G$ with the following properties: \begin{cond} \label{conditions2} \begin{enumerate} \item $G$ and $c$ satisfy Conditions \ref{conditions1} above. \item All ramification of $L/K$ is of type $c$. \item $K$ and $L$ are totally real number fields. \end{enumerate} \end{cond}

In the case where $G=A_5$ and $K=\QQ$, if $c$ is the conjugacy class of 3-cycles so that $H_2(G,c,\ZZ)\cong C_2$, we can define the invariant in more down-to-earth terms. In this case, $\widetilde{G}_c=\widetilde{A}_5$ and $\widetilde{c}$ is the conjugacy class of elements of order 3 in $\widetilde{G}_c$. 

\begin{Claim} If $L/\QQ$ is the $A_5$ field which is the fixed field of the kernel of $\rho$ and $\widetilde{L}$ is the fixed field of the kernel of $\widetilde{\rho}$, then the invariant $\mfz(\rho)$ is the number of primes $p\equiv 3\pmod 4$ with even ramification index in $\widetilde{L}$, modulo 2. \end{Claim}

\begin{proof} We check the contribution to (\ref{invtdef}) at each prime. If $\widetilde{L}/L$ is unramified above $v$, then the contribution to the product is $1\in\{\pm 1\}$. Now suppose that $\widetilde{L}/L$ is ramified above $v$. We claim that if $v\equiv 3\pmod 4$, then the contribution to the product is $-1$, and if $v\equiv 1\pmod 4$, then the contribution to the product is $+1$. To see this, let $\sigma$ be a topological generator of $I_v^{\tame}$ for which $g_v=\sigma^{\frac{q_v-1}{2}}$. Then for each prime $\mfP$ of $L$ above $v$, $\rho(\sigma)$ is a generator for the inertia group $I_\mfP$ at $\mfP$, and $\rho(g_v)=\rho(\sigma)^{\frac{q_v-1}{2}}$. Furthermore, if $\msP$ is the prime of $\widetilde{L}$ above $\mfP$, then $\widetilde{\rho}(\sigma)$ is a generator for $I_\msP$, and $\widetilde{\rho}(g_v)=\widetilde{\rho}(\sigma)^{\frac{q_v-1}{2}}$. Also, note that the order of $I_\msP$ is twice that of $I_\mfP$, so the order of $I_\msP$ is congruent to $2\pmod 4$ (since we are assuming that all ramification in $L$ is of type a 3-cycle). Now, the $r^\text{th}$ power of a generator of a cyclic group of order congruent to $2\pmod 4$ has even order if and only if $r$ is odd. Hence, $\widetilde{\rho}(g_v)$ has even order if $v\equiv 3\pmod 4$ and odd order if $v\equiv 1\pmod 4$. Note that there are two elements in $\widetilde{A}_5$ which map to $\rho(g_v)\in A_5$, with one having even order and one having odd order. Hence, the contribution to the product is $+1$ if $\widetilde{\rho}(g_v)$ has odd order and $-1$ if $\widetilde{\rho}(g_v)$ has even order. Hence, the claim is valid, and so the product is $-1$ to the power of the number of primes congruent to $3\pmod 4$ which ramifiy in $\widetilde{L}/L$. Since all primes in $L$ have odd ramification degree and all ramified primes in $\widetilde{L}/L$ have ramification degree 2, the full claim holds. \end{proof}

\begin{rem} The invariant is independent of the choice of $\widetilde{L}$. Suppose we have another lift $\widetilde{L}'$. Then $\widetilde{L}$ and $\widetilde{L}'$ differ by a totally real quadratic twist $G_{\QQ}\to C_2$ unramified at 2, and in any real quadratic field unramified at 2, the number of ramified primes congruent to $3\pmod 4$ is even. \end{rem}

In the case where $G=A_4$ and $K=\QQ$, we can take $c$ to be the union of two conjugacy classes consisting of all 3-cycles of $G$. Then $H_2(G,c,\ZZ)\cong C_2$, and we take $\widetilde{G}_c=\widetilde{A}_4$ and $\widetilde{c}$ the collection of elements of order 3 in $\widetilde{G}_c$. The invariant is defined just like in the case of $G=A_5$ above: if $L/\QQ$ is the $A_4$ field over $\QQ$ which is the fixed field of a homomorphism $\rho:G_{\QQ}\to A_4$, we can lift to an $\widetilde{A}_4$ field $\widetilde{L}$ which is tamely ramified and totally real. The invariant $\mfz(\rho)$ is again the number of primes $p\equiv 3\pmod 4$ with even ramification index in $\widetilde{L}$, modulo 2.

Much of the value of the invariant rests on our belief in the following conjecture:

\begin{conj} Assume $G$ and $c$ are such that Conditions \ref{conditions1} hold. Consider all $L/K$ with $\Gal(L/K)\cong G$ satisfying Conditions \ref{conditions2}, with the norm of the discriminant of $L$ lying in the interval $[-N,N]$. Then as we let $N$ tend to $\infty$, $\mfz(\rho)$ is equidistributed over $H_2(G,c,\ZZ)$. \end{conj}

\begin{rem} We can think about invariants for $A_4$ fields unramified over their cyclic cubic subfield in one of two ways. First, of course, they are invariants of $A_4$ fields.  But such $A_4$ fields are in natural bijection with Galois cubic fields $K$ together with Galois-equivariant surjection $\varphi:\Cl(K)\twoheadrightarrow C_2\times C_2$, up to an equivalence relation on the surjections: given such a field $K$ and a surjection $\varphi$, we construct an unramified $C_2\times C_2$ cover $H$ of $K$, so that $H$ is Galois over $\QQ$ (and hence $K$), with $\Gal(H/K)\cong C_2\times C_2$ and $\Gal(H/\QQ)\cong A_4$. (This construction is described in section \ref{sec:comp}, and an example is given in detail in section \ref{sec:samp}.) Now, let $c$ be the trivial conjugacy class in $C_2\times C_2$. In this case, we have $H_2(G,c,\ZZ)\cong C_2$, so $H$ lifts to fields $\widetilde{H}_1$ and $\widetilde{H}_2$, with $\Gal(\widetilde{H}_1/K)\cong D_8$ and $\Gal(\widetilde{H}_2/K)\cong Q_8$; Lemma \ref{EVlemma} also allows us to take $\widetilde{H}_1/K$ and $\widetilde{H}_2/K$ to be tamely ramified and totally real. Since the Sylow 2-subgroup of $\widetilde{A}_4$ is isomorphic to $Q_8$, $\widetilde{H}_2$ is an $\widetilde{A}_4$-field. If, furthermore, all ramification in $H$ is of type 3-cycle, then $\widetilde{H}_2$ is a lift of $H$ of the type described above. Hence, we can also think of the invariant associated to an $A_4$ field $H$ as being the invariant associated to a pair $(K,\varphi)$, where $K$ is a $C_3$ field and $\varphi$ a surjection from $\Cl(K)$ to $C_2\times C_2$. Thus, the invariant associated to $H$ is either the invariant associated to the representation $G_{\QQ}\to A_4$ with fixed field $H$, or the invariant associated to the representation $G_K\to C_2\times C_2$, again with fixed field $H$; we get the same result in $C_2$ in either case. \end{rem}

\section{Computing the Invariant} \label{sec:comp}

In this section, we present an algorithm that takes an $A_4$ field ramified at one prime and produces an $\widetilde{A}_4$ lift of it. We then prove that the invariant associated to an $A_4$ field is closely related to the class group; this will help us compute tables of invariants much more quickly than if we had to construct the $\widetilde{A}_4$ field in every case.

Let $c$ be set of all 3-cycles in $A_4$; this is a union of two conjugacy classes in $A_4$. If $H$ is ramified at exactly one prime, we will prove below that all ramification in $H$ is of type $c$, so Lemma \ref{EVlemma} tells us that we can lift $H$ to a tamely ramified and totally real extension $\widetilde{H}$ with Galois group $\widetilde{A}_4\cong\SL_2(\FF_3)$. If $H$ is ramified at more than one prime, all we can say is that the ramification of \emph{some} prime is of type $c$.

\begin{prop} \label{A4ram} If $E/\QQ$ is an $A_4$ field ramified at exactly one rational prime, then all ramification is of type 3-cycle. \end{prop}

\begin{rem} Over other base fields, this Proposition is false. For example, if we let $K=\QQ(\sqrt{-59})$ and $L/K$ the extension given by the polynomial $x^4-x^3-7x^2+11x+3$, then $L/K$ is an $A_4$-extension ramified only at the (unique) prime $\mfp$ of $K$ above 59, but $\mfp$ factors as $\mfP^2$, where $\mfP$ is a prime ideal of $L$. \end{rem}

This Proposition follows quickly from the following more general Lemma:

\begin{Lemma} \label{inertia} If $E/\QQ$ is a finite Galois extension with Galois group $G$, then the inertia groups at the ramified finite places of $E$ generate $G$. \end{Lemma}

\begin{proof} Let $E_0$ be the intersection of the fixed fields of all the inertia groups. Then $E_0/\QQ$ is a finite Galois extension unramified at all finite places. Hence $E_0=\QQ$, and so the inertia groups generate $G$. \end{proof}

\begin{proof}[Proof of Proposition \ref{A4ram}] There are no wildly ramified $A_4$ extensions of $\QQ$ ramified at exactly one rational prime (see \cite{JonesNF}), so $E$ must be tamely ramified. Hence, its ramification type must either be that of 3-cycles, or that of products of two disjoint 2-cycles. The latter case cannot happen by Lemma \ref{inertia}, because the products of two disjoint 2-cycles do not generate $A_4$. \end{proof}

Now, we shall see how to lift a totally real $A_4$ field $H$ ramified at exactly one rational prime to a tamely ramified and totally real $\widetilde{A}_4$ field $\widetilde{H}$. Note that, since $\widetilde{H}/H$ will be a quadratic extension, tamely ramified is equivalent to being unramified above 2.

\begin{algorithm}{Making an $\widetilde{A}_4$ extension of $H$}{\label{A4tilde-algo} \qinput{A quartic polynomial $f$ defining a quartic field $L$ with Galois closure a totally real $A_4$ field $H$ ramified at exactly one prime.} \qoutput{An element $\alpha\in L$ so that $L(\sqrt{\alpha})$ has Galois closure an $\widetilde{A}_4$ field $\widetilde{H}$, and so that $\widetilde{H}$ is totally real and tamely ramified.}} Let $\{\alpha_i\}$ be a set of representatives of $\mfo_L^\times/\mfo_L^{\times 2}$ which includes 1. \\ Let $\{C_j\}_{j\in J}$ be the 2-torsion ideal classes of $L$. \\ For $j\in J$, let $I_j$ denote an integral ideal in $C_j$. Let the ideal $(1)$ be the representative of the trivial ideal class. \\ Each $I_j^2$ is a principal ideal; let $\beta_j$ be a generator for $I_j^2$. \\ Let $\gamma_1=1$. \\ Let $p$ be the rational prime at which $L$ is ramified. Then $p\mfo_L$ splits as $\mfp_1\mfp_2^3$, for some prime ideals $\mfp_1,\mfp_2\subset\mfo_L$. Let $J=\mfp_1\mfp_2$. Suppose that the order of $J$ in the class group is $r$. Let $\gamma_2$ be a generator for the principal ideal $J^r$. \\ Let $\Delta=\{\alpha_i\beta_j\gamma_k\}$. \\ For $\delta\in\Delta$, check if $L(\sqrt{\delta})$ has Galois group $\widetilde{A}_4$. Stop once we have found one that does, and call this element $\delta$. \label{findadelta} \\ If $L(\sqrt{\delta})$ is tamely ramified and totally real, let $\alpha=\delta$. \\ If $L(\sqrt{\delta})$ is tamely ramified and totally complex, let $q$ be a rational prime with $q\equiv 3\pmod 4$ so that $L(\sqrt{\delta})$ is unramified at $q$. Let $\alpha=-q\delta$. \\ If $L(\sqrt{\delta})$ is wildly ramified and totally real, let $q$ be a rational prime with $q\equiv 3\pmod 4$ so that $L(\sqrt{\delta})$ is unramified at $q$. Let $\alpha=q\delta$. \\ If $L(\sqrt{\delta})$ is wildly ramified and totally complex, let $\alpha=-\delta$. \\ Return $\alpha$. \end{algorithm}

\begin{proof}[Proof of Algorithm \ref{A4tilde-algo}] By Lemma \ref{EVlemma}, we know that there is an $\alpha\in H$ so that $H(\sqrt{\alpha})$ is Galois over $\QQ$ with Galois group $\widetilde{A}_4$. Suppose we have such an $\alpha$. Let $q$ be a rational prime different from $p$, and let $q\mfo_H=\mfq_1\cdots\mfq_g$. In order for $H(\sqrt{\alpha})$ to be Galois over $\QQ$ it is necessary and sufficient that the class of $\alpha$ in $H^\times/H^{\times 2}$ is stable under the action of $\Gal(H/\QQ)$. If the class of $\alpha$ in $H^\times/H^{\times 2}$ is $\Gal(H/\QQ)$-stable, then the parity of $v_{\mfq_i}(\alpha)$ must be the same for all $i$. If $v_{\mfq_i}(\alpha)\equiv 1\pmod 2$ for all $i$ and the class of $\alpha$ is $\Gal(H/\QQ)$-stable, then $v_{\mfq_i}(\alpha/q)\equiv 0\pmod 2$ for all $i$, and the class of $\alpha/q$ is still $\Gal(H/\QQ)$-stable. Hence, we may assume that $v_{\mfq}(\alpha)\equiv 0\pmod 2$ for all primes $\mfq$ of $H$ lying over a rational prime different from $p$. Furthermore, the parities of $v_{\mfp_i}(\alpha)$ are equal for all primes $\mfp_i$ of $H$ lying over $p$. Let $\Xi$ be the set of square classes of $H$ with even valuation at all primes not lying over $p$, and with all valuations at primes over $p$ having the same parity.

Let $B$ be the kernel of the map $\widetilde{A}_4\twoheadrightarrow A_4$. Then $\widetilde{A}_4$ acts transitively and faithfully on a set of 8 objects partitioned into blocks of size 2 so that $B$ fixes the blocks. The quotient $A_4\cong\widetilde{A}_4/B$ acts on the blocks in the usual way that $A_4$ acts on 4 objects. Hence any $\widetilde{A}_4$ field is the Galois closure of an octic field obtained by adjoining the square root of some square class in a quartic field. Hence, we may restrict our list of square classes to check still further by letting $\Delta$ be the set of square classes in $\Xi$ which contain a representative in $L$. This shows that we can find a $\delta$ in Step \ref{findadelta} so that $L(\sqrt{\delta})$ has Galois group $\widetilde{A}_4$.

%We now observe that, since $\widetilde{A}_4$ can be embedded in $S_8$, any $\widetilde{A}_4$ field is the Galois closure of an octic field, and even an octic field obtained by adjoining the square root of some square class in a quartic field. Hence, we may restrict our list of square classes to check still further by letting $\Delta$ be the set of square classes in $\Xi$ defined over $L$. This shows that we can find a $\delta$ in Step \ref{findadelta} so that $L(\sqrt{\delta})$ has Galois group $\widetilde{A}_4$.

The remaining steps explain how we can twist by a quadratic character in order to remove wild ramification and ramification at $\infty$. Let $\widetilde{\rho}:G_{\QQ}\to\widetilde{A}_4$ be the Galois representation corresponding to the number field $H(\sqrt{\delta})$. We can find some quadratic character $\chi:G_{\QQ}\to C_2$ so that the representation $\chi\widetilde{\rho}$ is tamely ramified and unramified at $\infty$. This completes the proof. \end{proof}

Once we have found the desired $\widetilde{L}$, we can simply count the number of ramified primes congruent to $3\pmod 4$ in $\widetilde{L}$ in order to determine the invariant $\mfz(\rho)$.

Frequently, it is possible to compute the invariant without constructing an explicit lift to an $\widetilde{A}_4$ field. (Still, as a matter of good discipline and for the sake of generality, it is good to know how to perform the explicit construction.) We recall that $\Cl_2(K)$ denotes the Sylow 2-subgroup of $\Cl(K)$. We define variants such as $\Cl^+_2(K)$ to mean the Sylow 2-subgroup of $\Cl^+(K)$, and in general, a subscript of 2 in any sort of class group will denote the Sylow 2-subgroup of that class group. In the situation at hand, we have the following characterization of the invariant:

\begin{thm} \label{cleandescription} Let $K$ be a $C_3$ field with prime conductor so that $\Cl_2(K)\cong C_2\times C_2$, and let $H$ be the everywhere unramified $C_2\times C_2$ extension of $K$, so that $H$ is an $A_4$ field. Let $L\subset H$ be the fixed field of any 3-cycle in $\Gal(H/\QQ)$, so that $L$ is a non-Galois quartic field over $\QQ$ with Galois closure $H$. Suppose furthermore that $\Cl_2^+(L)$ is cyclic. Then the invariant associated to $H$ is 1 if $\Cl_2(L)$ is trivial and 0 if $\Cl_2(L)$ is nontrivial. \end{thm}

%\textcolor{red}{Is it possible to rephrase this theorem? Perhaps restrict to the 2-torsion of $\Cl(L)$, so that we can say that if the 2-torsion is trivial, then the invariant is 1, and if the 2-torsion is nontrivial, then the invariant is 0.}

%\begin{thm} \label{cleandescription} Let $K$ be a $C_3$ field with prime conductor so that $\Cl_2(K)\cong C_2\times C_2$, and let $H$ be the everywhere unramified $C_2\times C_2$ extension of $K$ (so that $H$ is an $A_4$ field). \begin{enumerate} \item $\Cl^+_2(H)/2\Cl^+_2(H)\cong C_2$. \label{narclgp} \item Let $L\subset H$ be the fixed field of any 3-cycle in $\Gal(H/\QQ)$, so that $L$ is a non-Galois quartic field over $\QQ$ with Galois closure $H$. Then the invariant associated to $H$ is equal to $\#\Cl(L)\pmod 2$. \end{enumerate} \end{thm}

\begin{proof} By the above, we have an $\widetilde{A}_4$ field $\widetilde{H}$ containing $H$ so that $\widetilde{H}$ is totally real and tamely ramified. While the construction of $\widetilde{H}$ is not unique, any two such $\widetilde{H}$'s differ only by a quadratic twist $\chi:G_{\QQ}\to C_2$. Let $\widetilde{\rho}:G_{\QQ}\to\widetilde{A}_4$ be the Galois representation associated to one such $\widetilde{H}$. Suppose that $\widetilde{H}/H$ were ramified at two primes of $H$ above two distinct primes $p_1$ and $p_2$ of $\QQ$, so that $p_1\equiv p_2\equiv 3\pmod 4$. Let $\chi:G_{\QQ}\to C_2$ be the quadratic character associated to the number field $\QQ(\sqrt{p_1p_2})$. Then the number field associated to the representation $\chi\widetilde{\rho}$ is again an $\widetilde{A}_4$ field which is still tamely ramified, totally real, and contains $H$, and the corresponding quadratic extension of $H$ is ramified at exactly the primes at which $\widetilde{H}$ ramifies, with the exception of the primes above $p_1$ and $p_2$, where it is now unramified. Similarly, if $\widetilde{H}/H$ were ramified at a prime in $H$ above some rational prime $q\equiv 1\pmod 4$, then if $\chi$ is the quadratic character associated to $\QQ(\sqrt{q})$, then the number field associated to the representation $\widetilde{\rho}\chi$ is now unramified at the primes above $q$. Hence, we may assume that there are no primes congruent to $1\pmod 4$ for which the primes in $H$ above $p$ are ramified in $\widetilde{H}/H$, and there is at most one such prime congruent to $3\pmod 4$. If we have a prime $q\equiv 3\pmod 4$ for which the primes above $q$ are ramified in $\widetilde{H}/H$, let $\chi$ be the quadratic character associated to $\QQ(\sqrt{-q})$. Then the field associated to $\chi\widetilde{\rho}$ is unramified at the primes above $q$ in $H$.

The above paragraph shows us how to produce a quadratic extension $\widetilde{H}/H$ unramified at all finite places, so that $\Gal(\widetilde{H}/\QQ)\cong\widetilde{A}_4$. By the argument in the proof of Algorithm \ref{A4tilde-algo}, $\widetilde{H}$ descends to a quadratic extension $\widetilde{L}$ of $L$, unramified at all finite places, so that the Galois closure of $\widetilde{L}$ over $\QQ$ is $\widetilde{H}$. Hence, $\Cl_2^+(L)$ is nontrivial. If $\Cl_2^+(L)$ is cyclic, then there is a \emph{unique} nontrivial quadratic extension of $L$ unramified at all finite places, so this extension must be $\widetilde{L}$. In this case, $\widetilde{L}$ and hence $\widetilde{H}$ are totally real if and only if $\Cl_2(L)$ is nontrivial. If this happens, then $\widetilde{H}/H$ is everywhere unramified (including at infinity), so the invariant is 0. If $\Cl_2(L)$ is trivial, then $\widetilde{H}/H$ is ramified only at infinity, so we can twist by some character associated to a field $\QQ(\sqrt{-p})$ for some $p\equiv 3\pmod 4$ to obtain a totally real $\widetilde{H}$ ramified only at $p$. Hence, the invariant in this case is 1. In either case, the invariant is $\#\Cl_2(L)\pmod 2$. \end{proof}

\begin{rem} The hypothesis that $\Cl_2^+(L)$ is cyclic holds very frequently. In fact, there are no exceptions in the fields we tested for inclusion in the data given in Table \ref{invtdata}. One might suspect that the hypothesis always holds, but we were not able to give a proof, or to come up with a counterexample. \end{rem}

%\textcolor{red}{It might be worth asking whether this is always the case, but it's a little risky to formulate this as a full-out conjecture.}

For our computations, we will need to start with a $C_3$ field $K$ with $\Cl_2(K)\cong C_2\times C_2$ and construct an $A_4$ field $H$ containing $K$ so that $H/K$ is everywhere unramified. We now explain how that is done.

\begin{algorithm}{Constructing an $A_4$ field from a $C_3$ field $K$ with $\Cl_2(K)\cong C_2\times C_2$}{\label{A4-algo}\qinput{A cubic polynomial $f$ defining a Galois cubic field $K$.} \qoutput{A quartic polynomial $g$ so that the Galois closure of $g$ is an $A_4$ field $H$ containing $K$, with $H/K$ everywhere unramified.}} Let $\{\alpha_i\}$ be a set of representatives of $\mfo_K^\times/\mfo_K^{\times 2}$. \\ Let $\{C_j\}_{j\in J}$ be the 2-torsion ideal classes of $K$. \\ For $j\in J$, let $I_j$ denote an integral ideal in $C_j$. Let the ideal $(1)$ be the representative of the trivial ideal class. \\ Each $I_j^2$ is a principal ideal; let $\beta_j$ be a generator for $I_j^2$. \\ Let $\Delta=\{\alpha_i\beta_j\}$. \\ For $\delta\in\Delta$, let $K_\delta$ be the Galois closure over $\QQ$ of $K(\sqrt{\delta})$. If $K_\delta$ has Galois group $A_4$ and is totally real and unramified at 2, let $\alpha=\delta$ and stop. \\ Let $x^3-a_2x^2+a_1x-a_0$ be the minimal polynomial of $\alpha$ over $\QQ$. \\ Let $b_2=-2a_2$, $b_1=-8\sqrt{a_0}$, $b_0=a_2^2-4a_1$. \\ Let $h(x) = x^4+b_2x^2+b_1x+b_0$. \\ (Optional.) Using the LLL algorithm, find a polynomial $g$ with smaller coefficients than $h$ so that $g$ and $h$ generate the same field; this is implemented in PARI/GP \cite{PARI} as {\tt{polredabs}}. \\ Return $g$. \end{algorithm}

\begin{proof}[Proof of Algorithm \ref{A4-algo}] We first show that there is some $\alpha\in\Delta$ so that the Galois closure $H$ of $K(\sqrt{\alpha})$ has Galois group $A_4$ over $\QQ$ and so that $H/K$ is everywhere unramified. By class field theory, we know that there is some such $\alpha\in K^\times$, so it suffices to show that if $\alpha\not\in\Delta K^{\times 2}$, then $K(\sqrt{\alpha})/K$ is ramified somewhere. Observe that for $K(\sqrt{\alpha})/K$ to be unramified, it is necessary (but not sufficient) that $\alpha$ have even valuation at all places of $K$. Those elements of $K^\times$ which have even valuation at all places of $K$ are precisely the elements of $\Delta K^{\times 2}$, so this shows that we can find such an $\alpha\in\Delta$.

Now we explain the construction of $h(x)$. The $A_4$ field $H$ is the Galois closure of a quartic field $L$ over $\QQ$. Since $\alpha$ has degree 3 over $\QQ$ and is not a square, $\sqrt{\alpha}$ has degree 6. Let us call its Galois conjugates $\pm\sqrt{\alpha},\pm\sqrt{\beta},\pm\sqrt{\gamma}$. Now, $L$ is generated by $r=\sqrt{\alpha}+\sqrt{\beta}+\sqrt{\gamma}$. The conjugates of $r$ are $r_2=\sqrt{\alpha}-\sqrt{\beta}-\sqrt{\gamma}$, $r_3=-\sqrt{\alpha}+\sqrt{\beta}-\sqrt{\gamma}$, and $r_4=-\sqrt{\alpha}-\sqrt{\beta}+\sqrt{\gamma}$, and so we can check explicitly that $h$ as constructed in Algorithm \ref{A4-algo} is the minimal polynomial of $r$. \end{proof}

%\textcolor{red}{I don't know where the best place is for this part. Maybe a new section?} Let us describe this construction. \textcolor{red}{Write this as an algorithm.} First, let us describe how to construct $H(K)$. We know that $H(K)$ is the Galois closure over $\QQ$ of a quadratic extension of $K$, so by Kummer theory it suffices to find a suitable $\alpha\in K^\times/K^{\times 2}$ of which to adjoin a square root. For convenience, we will sometimes abuse notation by writing the same letter for a square class and a representative of it; we hope that this will not be confusing. Our method for finding an appropriate $\alpha$ is to produce a relatively short list of possibilities, (at least) one of which is guaranteed to work. In order for $K(\sqrt{\alpha})/K$ to be an everywhere unramified extension, it is necessary that $\alpha$ have even valuation at all finite places $v$ of $K$, but this is not sufficient, as $K(\sqrt{\alpha})/K$ may still be ramified at 2 and $\infty$. However, this necessary condition allows us to generate a finite list of possibilities: $\alpha$ can be written as $\alpha=\alpha_1\alpha_2$, where \begin{enumerate} \item $\alpha_1\in\mfo_K^\times/\mfo_K^{\times 2}$. \item $\alpha_2$ is such that $(\alpha_2)$ is the square of an ideal of $K$. It suffices to consider one possible $\alpha_2$ for each 2-torsion ideal class of $K$. \textcolor{red}{This isn't precise; currently case 1 is contained in case 2, but that isn't how I mean it.} \end{enumerate}

Suppose now that $K$ is a $C_3$ field ramified at exactly one rational prime $p$. Then $H$ as constructed in Algorithm \ref{A4-algo} is also ramified at $p$ and nowhere else. Furthermore, $H$ is totally real.

\section{A Sample Invariant Computation} \label{sec:samp}

Let us compute the invariant for the smallest $A_4$ field ramified at one prime. In order to build this $A_4$ field, we start with the smallest $C_3$ field with prime conductor and class group $C_2\times C_2$. A polynomial generating this field is $p(x)=x^3-x^2-54x+169$; the field is ramified only at 163. (The fact that this is the smallest such $A_4$ field seems not to be related to the most famous fact about the number $163$ and class groups, namely that $-163$ is the discriminant of the largest imaginary quadratic field with class number 1. However, it \emph{is} related to the other interesting property of 163: that 163 is the \emph{smallest} $p$ for which the class number of $\QQ(\zeta_p+\zeta_p^{-1})$ has class number greater than 1; see \cite{Mcl10}.) Let $K$ denote this field, and let $\alpha$ be a root of $p$ in $K$. The unit group is \[\mfo_K^\times=(\alpha-4)^{\ZZ}\times(\alpha^2+4\alpha-33)^{\ZZ}\times\{\pm 1\}.\] Two ideals whose ideal classes generate the class group are $(5,\alpha-2)$ and $(5,\alpha-1)$, and the squares of these ideals are $(\alpha^2+4\alpha-32)$ and $(\alpha^2+4\alpha-35)$, respectively.

To find the Hilbert class field of $K$, it suffices to look at $K(\sqrt{\beta})$, where $\beta$ is the product of elements of some subset of $\{-1,\alpha-4,\alpha^2+4\alpha-33,\alpha^2+4\alpha-32,\alpha^2+4\alpha-35\}$. We find that, if $\beta$ is one of $\{\gamma_1,\gamma_2,\gamma_3\}$, where $\gamma_1=\alpha^2+4\alpha-32$, $\gamma_2=(\alpha-4)(\alpha^2+4\alpha-33)(\alpha^2+4\alpha-35)=12\alpha^2+48\alpha-395$, and $\gamma_3=(\alpha-4)(\alpha^2+4\alpha-33)(\alpha^2+4\alpha-32)(\alpha^2+4\alpha-35)=169\alpha^2+688\alpha-5612$, then $K(\sqrt{\beta})$ is an unramified extension of $K$. Hence, if $\beta$ is one of these elements, then the Hilbert class field $H(K)$ of $K$ is the Galois closure of $K(\sqrt{\beta})$.

Now, we find an equation for a quartic field $L$ whose Galois closure is $H(K)$. Take $\beta$ as in the above paragraph, so that $H(K)$ is the Galois closure of $K(\sqrt{\beta})$. Suppose the Galois conjugates of $\beta$ in $K$ are $\beta_1=\beta,\beta_2,\beta_3$. Then the minimal polynomial of $\sqrt{\beta}$ over $\QQ$ has roots $\pm\sqrt{\beta_1}$, $\pm\sqrt{\beta_2}$, and $\pm\sqrt{\beta_3}$. An example of a quartic polynomial with the same Galois closure has roots $\sqrt{\beta_1}+\sqrt{\beta_2}+\sqrt{\beta_3}$, $\sqrt{\beta_1}-\sqrt{\beta_2}-\sqrt{\beta_3}$, $-\sqrt{\beta_1}+\sqrt{\beta_2}-\sqrt{\beta_3}$, and $-\sqrt{\beta_1}-\sqrt{\beta_2}+\sqrt{\beta_3}$. A polynomial with these roots is $x^4-34x^2-40x+121$. Using the PARI/GP function {\tt polredabs}, we find that another polynomial that generates the same field is $f(x)=x^4-x^3-7x^2+2x+9$. Let $L$ be the field $\QQ[x]/(f(x))$, and let $\gamma$ be a root of $f$ in $L$.

Now, we must construct a degree-8 extension of $L$ whose Galois group is isomorphic to $\widetilde{A}_4$. First, we must find generators for the unit group of $L$. The unit group of $L$ is isomorphic to $\ZZ^3\times\{\pm 1\}$, and a basis for the torsion-free part is $\{\gamma^2-2,\gamma+2,\gamma^2-2\gamma-4\}$. The class group of $L$ is trivial, so we get no contribution from 2-torsion ideal classes. Finally, $L$ is ramified exactly at the prime 163, and $(163)$ factors as $\mfp_1\mfp_2^3$, where $\mfp_1=(\gamma^3-4\gamma-4)$ and $\mfp_2=(-4\gamma^3+9\gamma^2+16\gamma-26)$. Hence, $\mfp_1\mfp_2=(6\gamma^3-11\gamma^2-23\gamma+23)$. Thus, some field of the form $L(\sqrt{\delta})$, where $\delta$ is the product of elements of some subset of $\{\gamma^2-2,\gamma+2,\gamma^2-2\gamma-4,6\gamma^3-11\gamma^2-23\gamma+23\}$, has Galois group $\widetilde{A}_4$. We find that, if $\delta=\gamma+2$, then the Galois closure of $L_1=L(\sqrt{\delta})$ has Galois group $\widetilde{A}_4$ over $\QQ$. Now, $L_1$ is totally real, but it is ramified at 2, and hence not tamely ramified. Thus, we need to twist by a character that is ramified at one prime congruent to $3\pmod 4$. In particular, $L_2=L(\sqrt{3\delta})$ has Galois group $\widetilde{A}_4$, is totally real, and is tamely ramified, so it is a lift of the desired form. Now, in order to compute $\mfz(\rho)$, we need to determine the number of primes congruent to $3\pmod 4$ that are ramified to even order. There are two ramified primes of $L_2$, namely 3 and 163. Only 3 is ramified to even order, so $\mfz(\rho)=1$.

This computation, and all others in this paper, were done using Sage \cite{sage} and PARI/GP \cite{PARI}.

\section{The Data}

We collected data from all the $C_3$ fields $K$ ramified at exactly one prime below $10^8$ such that the Sylow 2-subgroup of the class group is $K$ is isomorphic to $C_2\times C_2$ and constructed the associated $A_4$ fields. There were 377550 such fields. However, in 21 cases, the computation was prohibitively lengthy due to the fact that many small primes are inert; therefore, we skipped these cases. Their absence does not make any substantial difference to the evidence of equidistribution, or lack thereof. Of the 377529 fields considered, 202647 have invariant 1 and 174882 have invariant 0. So, while we might be forgiven for expecting that the invariant equidistributes among the two classes, the data seem to exhibit a slight bias that gradually goes away. Table~\ref{invtdata} gives incremental data for the invariants. The numbers in the third column are decreasing, but based on the trends in the observed data, it does not appear that our current computational abilities allow us to observe them getting very close to $1/2$.

\begin{table}[ht]\caption{Invariant Data}{\begin{tabular}{ccc} $N$ & Invariant 1 & Proportion with invariant 1 \\ \hline 100 & 55 & .5500 \\ 200 & 104 & .5200 \\ 300 & 160 & .5333 \\ 400 & 212 & .5300 \\ 500 & 266 & .5320 \\ 1000 & 536 & .5360 \\ 2000 & 1063 & .5315 \\ 4000 & 2183 & .5458 \\ 6000 & 3279 & .5465 \\ 8000 & 4370 & .5463 \\ 10000 & 5457 & .5457 \\ 20000 & 10862 & .5431 \\ 30000 & 16267 & .5422 \\ 40000 & 21638 & .5410 \\ 50000 & 27064 & .5413 \\ 60000 & 32400 & .5400 \\ 70000 & 37768 & .5395 \\ 80000 & 43176 & .5397 \\ 90000 & 48578 & .5398 \\ 100000 & 53889 & .5389 \\ 150000 & 80717 & .5381 \\ 200000 & 107465 & .5373 \\ 250000 & 134348 & .5374 \\ 300000 & 161112 & .5370 \\ 350000 & 187918 & .5369 \\ 377529 & 202647 & .5368 \end{tabular}} \label{invtdata} \end{table}

The first column denotes the number of fields, the second denotes the number with invariant 1, and the third denotes the proportion with invariant 1. Hence, we suspect, somewhat hesitantly, that the two classes do equidistribute, but that there is a secondary term of slightly lower order that leads to an apparent bias that persists for a long time. Based on the numerical evidence, and the fact that the number of cubic fields with absolute value of the discriminant at most $x$ is of the form \[ax+bx^{5/6}+o(x^{5/6})\] for certain explicitly known constants $a$ and $b$ (see \cite{BST} and \cite{TT}), we might conjecture that the proportion of these $C_3$ fields with invariant 1 among the first $x$ by discriminant is \[1/2+cx^{-1/6}\log(x)+o(x^{-1/6}\log(x)),\] where $c\approx 0.024$; we expect a logarithmic term because we are parametrizing fields in a slightly different manner from \cite{BST} and \cite{TT}. Still, there is not yet enough data to be able to distinguish between an error term of the form $cx^{-1/6}\log(x)$ and, perhaps, $c'x^{-1/8}$, so this conjecture ought to be taken with more than a grain of salt.

%I computed all the $C_3$ fields ramified at one of the first 100000 primes; there are 49961 of them. (There is exactly one for each prime $p\equiv 1\pmod 3$.) Of these $C_3$ fields, in 6284 of them, the Sylow 2-subgroup of the class group is isomorphic to $C_2\times C_2$. I then computed the invariant for each of them: in 2846, the invariant is 0, and in the remaining 3438, the invariant is 1. So, we might be forgiven for conjecturing that, in the long run, there should be an equal number of $A_4$ fields with invariant 1 and with invariant 0, but that there is a secondary term of lower order that leads to a slight bias in favor of invariant 1.

\section*{Acknowledgments}

I would like to thank my advisor, Akshay Venkatesh, for suggesting this problem to me and for providing me with a large amount of insight and wisdom. In addition, I am grateful to Jordan Ellenberg, Silas Johnson, J\"urgen Kl\"uners, Hendrik Lenstra, Carl Pomerance, David Roberts, Craig Westerland, and the anonymous referee for their useful remarks and suggestions.

\bibliographystyle{plain}
\bibliography{A4invariants}

\end{document}